\documentclass[12pt]{amsart}
\usepackage[normalem]{ulem}

\usepackage{verbatim}
\usepackage{amssymb}
\usepackage{upref}
\usepackage[all]{xy}
\usepackage{color}

\allowdisplaybreaks

\newtheorem{thm}{Theorem}[section]
\newtheorem*{thm*}{Theorem}
\newtheorem{lem}[thm]{Lemma}
\newtheorem{cor}[thm]{Corollary}
\newtheorem{prop}[thm]{Proposition}

\theoremstyle{definition}
\newtheorem{defn}[thm]{Definition}

\newtheorem*{notn*}{Notation}

\newtheorem*{hyp*}{Hypothesis}

\newtheorem{rem}[thm]{Remark}
\newtheorem*{rem*}{Remark}

\numberwithin{equation}{section}

\newcommand{\secref}[1]{Section~\textup{\ref{#1}}}

\newcommand{\corref}[1]{Corollary~\textup{\ref{#1}}}
\newcommand{\lemref}[1]{Lemma~\textup{\ref{#1}}}
\newcommand{\propref}[1]{Proposition~\textup{\ref{#1}}}

\newcommand{\defnref}[1]{Definition~\textup{\ref{#1}}}

\newcommand{\midtext}[1]{\quad\text{#1}\quad}

\renewcommand{\and}{\midtext{and}}

\newcommand{\Z}{\mathbb Z}

\newcommand{\T}{\mathbb T}

\newcommand{\KK}{\mathcal K}

\newcommand{\LL}{\mathcal L}

\newcommand{\OO}{\mathcal O}

\newcommand{\Chi}{\raisebox{2pt}{\ensuremath{\chi}}}
\renewcommand{\epsilon}{\varepsilon}

\DeclareMathOperator{\aut}{Aut}

\DeclareMathOperator{\ad}{Ad}

\DeclareMathOperator{\ind}{Ind}

\DeclareMathOperator*{\spn}{span}
\DeclareMathOperator*{\clspn}{\overline{\spn}}

\newcommand{\id}{\text{\textup{id}}}
\newcommand{\rg}{\text{rg}}

\newcommand{\<}{\langle}
\renewcommand{\>}{\rangle}

\newcommand{\inv}{^{-1}}

\newcommand{\iso}{\overset{\cong}{\longrightarrow}}
\renewcommand{\bar}{\overline}

\newcommand{\wilde}{\widetilde}

\newcommand{\cs}%
{\ensuremath{\mathbf{C^*}}}
\newcommand{\csnd}%
{\ensuremath{\cs\!\!_\mathbf{nd}}}
\newcommand{\coact}%
{\ensuremath{\mathbf{C^*coact}}}
\newcommand{\coactnd}%
{\ensuremath{\coact_\mathbf{nd}}}
\newcommand{\coactn}%
{\ensuremath{\coact^\mathbf{n}}}
\newcommand{\coactnnd}%
{\ensuremath{\coactn_\mathbf{nd}}}
\newcommand{\coactm}%
{\ensuremath{\coact^\mathbf{m}}}
\newcommand{\coactmnd}%
{\ensuremath{\coactm_\mathbf{nd}}}

\newcommand{\cpct}[1]{#1^{(1)}}
\newcommand{\ideal}[2]{M(#1;#2)}
\newcommand{\cpcat}{\ensuremath{\mathsf{CPCorres}}}
\renewcommand{\csnd}{\ensuremath{\mathsf{C^*_{nd}}}}
\newcommand{\inn}[2]{\langle #1,#2 \rangle}

\begin{document}
\title{Functoriality of Cuntz-Pimsner correspondence maps}
\author[Kaliszewski, Quigg, and Robertson]
{S.~Kaliszewski, John Quigg, and David Robertson}
\address [S.~Kaliszewski] {School of Mathematical and Statistical Sciences, Arizona State University, Tempe, Arizona 85287} \email{kaliszewski@asu.edu}\
\address[John Quigg]{School of Mathematical and Statistical Sciences, Arizona State University, Tempe, Arizona 85287} \email{quigg@asu.edu}
\address[David Robertson]{School of Mathematics and Applied Statistics, University of Wollongong, NSW 2522, AUSTRALIA} \email{droberts@uow.edu.au}

\date{October 8, 2012}

\begin{abstract}
We show that the passage from a $C^*$-correspondence to its Cuntz-Pimsner $C^*$-algebra gives a functor on a category of $C^*$-correspondences with appropriately defined morphisms.  Applications involving topological graph $C^*$-algebras are discussed, 
and an application to crossed-product correspondences is presented in detail.
\end{abstract}

\subjclass[2010]{Primary 46L08; Secondary 46L55}
\keywords {Hilbert module, $C^*$-correspondence, Cuntz-Pimsner algebra, action, coaction}

\maketitle

\section{Introduction}

Cuntz-Pimsner algebras were first introduced by Pimsner in \cite{Pi} as $C^*$-algebras associated to $C^*$-correspondences with injective left actions; Katsura extended the definition in \cite{KatsuraCorrespondence} to include $C^*$-correspondences with non-injective left actions. The class of Cuntz-Pimsner algebras is very rich, containing all Cuntz-Krieger algebras, crossed products by $\Z$, and topological graph algebras.
Accordingly, there is a pressing need to understand how constructions at the level of $C^*$-correspondences carry over to the Cuntz-Pimsner algebras.

Now, many $C^*$-correspondence constructions naturally and necessarily involve
\emph{multiplier correspondences} $(M(Y),M(B))$.  
For example, if a group $G$ acts 
on a correspondence $(X,A)$, then $(X,A)$ embeds in the multipliers
$(M(X\rtimes G),M(A\rtimes G))$
of the crossed-product correspondence,
but not in general in $(X\rtimes G,A\rtimes G)$ itself.
Our main result (Corollary~\ref{functor}) is that a $C^*$-correspondence homomorphism
$(X,A) \to (M(Y),M(B))$ that
is \emph{covariant} in an appropriate sense (Definition~\ref{CP correspondence homomorphism})
induces a $C^*$-algebra homomorphism $\OO_X \to M(\OO_Y)$ of the corresponding Cuntz-Pimsner algebras.

Given that the natural embedding of a $C^*$-correspondence in its Cuntz-Pimsner algebra is often degenerate, there is no reason to expect that a correspondence
homomorphism $(X,A)\to (M(Y),M(B))$ will automatically extend to $(M(X),M(A))$; this makes composing two homomorphisms problematic. 
To deal with this, our notion of covariance incorporates the assumption that the image of $X$ lies in the so-called \emph{restricted multipliers} $M_B(Y)$ that were introduced in~\cite[Appendix~A]{dkq}. We then can show (Theorem~\ref{is a category}) that the $C^*$-homomorphisms $\OO_X \to M(\OO_Y)$ are obtained in a functorial way on an appropriately-defined category of $C^*$-correspondences.
Some earlier work has been done on functoriality of Cuntz-Pimsner algebras \cite{RobertsonExtensions, RobSzyCorrespondence}, but our approach is the first to incorporate multipliers.

In \secref{applications} we give three applications of our techniques: to topological graph actions, topological graph coactions, and to crossed-product $C^*$-correspondences. The first two are really just brief indications of applications, the first to a recent result \cite{dkq} concerning free and proper actions of groups on topological graphs, where the argument was essentially a ``bare-hands'' special case of our \corref{functor}, and the second is a foreshadowing of how we plan to use the technique to define coactions on Cuntz-Pimsner algebras from suitable coactions on the correspondences.
For the third application, we 
bring our methods to bear on crossed-product $C^*$-correspondences $(X \rtimes_\gamma G, A \rtimes_\alpha G)$. We show that the canonical embedding $(X,A) \to (M(X \rtimes_\gamma G),M(A \rtimes_\alpha G))$ induces via functoriality a covariant homomorphism $(\OO_X,G) \to M(\OO_{X \rtimes_\gamma G})$. As a result we find a surjective homomorphism $\OO_X \rtimes_\beta G \to \OO_{X \rtimes_\gamma G}$, which serves as an alternative approach to the result of Hao and Ng \cite[Theorem 2.10]{HN} that when the group $G$ is amenable there is an isomomorphism $\OO_{X \rtimes_\gamma G} \cong \OO_X \rtimes_\beta G$. By showing there is a maximal dual coaction on $\OO_{X \rtimes_\gamma G}$, we can show that our surjection is actually an isomorphism when the group $G$ is amenable, though this requires background on $C^*$-correspondence coactions and is saved for a forthcoming paper \cite{KQRCorrespondenceCoaction}.

\section{Preliminaries}

Given $C^*$-algebras $A$ and $B$, an \emph{$A-B$ correspondence} 
(or a \emph{$C^*$-correspondence over $A$ and $B$}) 
is a Hilbert $B$-module $X$ equipped with a left $A$-module action 
which is implemented by a homomorphism of~$A$ into 
the $C^*$-algebra $\LL(X)$ of adjointable operators on~$X$.
(We refer to \cite{enchilada, Ka1, tfb} for background on Hilbert modules
and further details on $C^*$-correspondences.)
The homomorphism is generically called $\varphi_A$, 
but we usually suppress this notation and just write 
$a \cdot \xi$ for $\varphi_A(a)(\xi)$, where $a \in A$ and $\xi \in X$. We say $X$ is \emph{full} if $\spn\langle X,X \rangle \subset B$ is dense. If in addition, the left action is an isomorphism $\varphi_A : A \to \KK(X)$ we will call $X$ an $A-B$ \emph{imprimitivity bimodule}.
We write $(A,X,B)$ to denote an $A-B$ correspondence $X$;
when $A=B$ we just write $(X,A)$, and call $X$ an \emph{$A$-correspondence} 
(or a \emph{$C^*$-correspondence over~$A$}).
\emph{In this paper we will make the standing assumption that all $C^*$-correspondences are \emph{nondegenerate} in the sense that $A\cdot X = X$.}

Given a $C^*$-correspondence $(A,X,B)$, 
the set $\KK(X) = \clspn \{ \theta_{\xi,\eta} \mid \xi,\eta\in X\}$
is a (closed) two-sided ideal in $\LL(X)$
called the \emph{compact operators} on $X$,
where by definition $\theta_{\xi,\eta}(\zeta) = \xi\cdot\< \eta , \zeta \>$ 
for $\zeta\in X$. 
The Banach space $\LL(B,X)$
of adjointable operators from $B$ to $X$ (where $B$ is viewed as a Hilbert $B$-module
in the natural way) becomes an $M(A)-M(B)$ correspondence when
equipped with the natural left action of $M(A)$, 
right action of $M(B)$,
and $M(B)$-valued inner product $\< m , n \> = m^*n$
all given by composition of operators.  
This is called the \emph{multiplier correspondence} of $X$,
and is denoted by $M(X)$. There is an embedding 
 $\xi \mapsto T_\xi:X \to M(X)$ given by $T_\xi(b) = \xi\cdot b$ for $b\in B$.
The \emph{strict topology} on $M(X)$ is generated by the seminorms
\[
m\mapsto \|T\cdot m\|\midtext{and}m\mapsto \|m\cdot b\|\midtext{for}T\in \KK(X), b\in B.
\]

A \emph{correspondence homomorphism} between two $C^*$-correspondences $(A,X,B)$ and $(C,Y,D)$
is a triple $(\pi,\psi,\rho)$ consisting of $C^*$-homomorphisms $\pi\colon A \to M(C)$ and $\rho\colon B \to M(D)$
and a linear map $\psi\colon X \to M(Y)$ satisfying
\begin{enumerate}
\item $\psi( a \cdot \xi) = \pi(a) \cdot \psi(\xi)$,
\item $\psi(\xi \cdot b) = \psi(\xi) \cdot \rho(b)$, and
\item $\rho( \< \xi , \eta \>) = \< \psi(\xi) , \psi(\eta) \>$.
\end{enumerate}
(Note that condition~(ii) follows from~(iii).)

A correspondence homomorphism
$(\pi,\psi,\rho)$ is \emph{nondegenerate} when 
$\pi$ and $\rho$ are nondegenerate $C^*$-homomorphisms
and $\psi(X) \cdot B = Y$.
In this case, by \cite[Theorem 1.30]{enchilada} there is a unique strictly continuous extension
$\bar\psi\colon M(X) \to M(Y)$ such that
\[
(\bar\pi,\bar\psi,\bar\rho)\colon (M(A),M(X),M(B))\to (M(C),M(Y),M(D))
\]
is a correspondence homomorphism, where $\bar\pi$ and $\bar\rho$
are the usual extensions of nondegenerate $C^*$-homomorphisms to multiplier algebras.

If $(X,A)$ is an $A$-correspondence and $B$ is a $C^*$-algebra,
a correspondence homomorphism of the form $(\pi,\psi,\pi)\colon (A,X,A) \to (B,B,B)$
(where $B$ is viewed as a $B$-correspondence in the natural way)
is called a \emph{Toeplitz representation} of $(X,A)$ in $B$, and is denoted $(\psi, \pi)$.
It is a critical observation, usually attributed to Pimsner \cite[Lemma~3.2]{Pi} (see also \cite[Lemma~2.2]{KajPinWatIdeal}), that a Toeplitz representation of a $C^*$-correspondence determines a homomorphism of the algebra of compact operators. 
Here we derive this fact (Corollary~\ref{toeplitz} below) from the more fundamental 
Proposition~\ref{compacts}
to emphasize that it is really a property of the underlying Hilbert module structure
on~$X$ and~$B$.

\begin{prop}\label{compacts}
Let $X$ and $Y$ be Hilbert modules over $C^*$-algebras $A$ and $B$, respectively,
and suppose $(\psi,\rho)\colon (X,A) \to (M(Y),M(B))$ is a correspondence homomorphism. Then there is a unique homomorphism $\cpct\psi\colon\KK(X)\to M(\KK(Y))$ such that
\[
\cpct \psi(\theta_{\xi,\eta})={\psi(\xi)\psi(\eta)^*}\midtext{for}\xi,\eta\in X.
\]
If $\psi(X)\subseteq Y$, then  $\cpct\psi(\KK(X))\subseteq \KK(Y)$,
with $\cpct \psi(\theta_{\xi,\eta})=\theta_{\psi(\xi),\psi(\eta)}$.
\end{prop}

\begin{proof}
Obviously there can be at most one such $\cpct\psi$.
For existence, first suppose $\psi(X)\subseteq Y$.  
Without loss of generality, we may assume that $X$ is full
(otherwise, replace $A$ by the closed span of $\<X,X\>$);
similarly, we may assume that $Y$ is full,
and that $\psi(X)=Y$ and $\rho(A)=B$.

Next, let $C=\KK(X)$, so that $X$ is a $C-A$ imprimitivity bimodule. 
Let $I=\ker\rho$. The Rieffel correspondence (see \cite[Definition 1.7]{enchilada}) induces an ideal $J = X\text{-}\ind I := \{c \in C: c X \subset XI\}$ so that $XI$ and $X/XI$ are $J-I$ and $C/J-A/I$ imprimitivity bimodules respectively.
Then the quotient maps $(q_C, q_X, q_A)$ comprise a surjective imprimitivity bimodule homomorphism
of $(C,X,A)$ onto $(C/J, X/XI, A/I)$, and there is an imprimitivity bimodule isomorphism
$(\pi, \wilde\psi, \wilde\rho)$ of $(C/J, X/XI, A/I)$ onto $(\KK(Y),Y,B)$ such that
$\wilde\psi(\xi + J) = \psi(\xi)$ for $\xi\in X$.  
Taking $\cpct\psi = \pi\circ q_C$, for $\xi,\eta\in X$ we have
\begin{align*}
\cpct\psi(\theta_{\xi,\eta}) 
&= \pi(q_C({}_C\<\xi,\eta\>)) = \pi ( {}_C\<\xi,\eta\> + J)
= \pi( {}_{C/J}\<\xi+J,\eta+J\>)\\
&= {}_{\KK(Y)}\<\wilde\psi(\xi+J),\wilde\psi(\eta+J)\>
= {}_{\KK(Y)}\<\psi(\xi),\psi(\eta)\> = \theta_{\psi(\xi),\psi(\eta)}.
\end{align*}

For the general case, apply the above to the Hilbert $M(B)$-module $M(Y)$,
and note that by \cite[Remark~A.10]{dkq} we have
$\KK(M(Y))\subset M(\KK(Y))$,
with ${}_{\KK(M(Y))}\<m,n\>=mn^*$ for $m,n\in M(Y)$.
\end{proof}

\begin{cor}
[\cite{Pi, KajPinWatIdeal}]
\label{toeplitz}
Let $(X,A)$ be a $C^*$-correspondence,
and let $(\psi,\pi)$ be a Toeplitz representation of $(X,A)$ in a $C^*$-algebra $B$. 
Then there is a unique homomorphism $\cpct\psi\colon\KK(X)\to B$ such that
\[
\cpct\psi(\theta_{\xi,\eta})=\psi(\xi)\psi(\eta)^*.
\]
\end{cor}

For a $C^*$-correspondence $(X,A)$, we follow Katsura's convention \cite{Ka1} and define an ideal $J_X$ of $A$ by
\[
 J_X = \{a \in A \mid 
 \text{$\varphi_A(a) \in \KK(X)$ and $ab = 0$ for all $b \in \ker(\varphi_A)$}\}.
\]
A Toeplitz representation $(\psi,\pi)$ of $(X,A)$ in $B$ is \emph{Cuntz-Pimsner covariant} if
\[
 \cpct \psi (\varphi_A(a)) = \pi(a)
 \quad\text{for $a\in J_X$.}
\]
We denote the universal Cuntz-Pimsner covariant Toeplitz representation by $(k_X,k_A)$, and the \emph{Cuntz-Pimsner algebra} by $\OO_X$.
Note that $k_A\colon A\to \OO_X$ will always be nondegenerate as a $C^*$-homomorphism 
because of our assumption that 
all correspondences $(X,A)$ are nondegenerate.
Again, we refer the reader to \cite{KatsuraCorrespondence} for details.

Finally,
we will need the theory of ``relative multipliers'' from \cite[Appendix~A]{dkq}, which is useful for extending degenerate homomorphisms.
If $(X,A)$ is a nondegenerate correspondence and $\kappa\colon C\to M(A)$ is a nondegenerate homomorphism, the
\emph{$C$-multipliers} of $X$ are by definition
\[
M_C(X)=\{m\in M(X) \mid \kappa(C)\cdot m\cup m\cdot \kappa(C)\subset X\}.
\]
The \emph{$C$-strict} topology on $M_C(X)$ is generated by the seminorms
\[
m\mapsto \|\kappa(c)\cdot m\|
\midtext{and}
m\mapsto \|m\cdot \kappa(c)\|
\quad\text{for $c\in C$.}
\]
When $A$ is viewed as an $A$-correspondence over itself in the usual way,
$M_C(A)$  is a $C^*$-subalgebra
of $M(A)$.

Relative multipliers possess the following elementary properties:
\begin{enumerate}
\item
The $C$-strict topology is stronger than the relative strict topology on $M_C(X)$.

\item
$M_C(X)$ is an $M_C(A)$-correspondence with respect to the restrictions of the operations of the $M(A)$-correspondence $M(X)$, and the operations are separately $C$-strictly continuous.

\item
If $X=A$, then $M_C(A)$ is a $C^*$-subalgebra of $M(A)$, and
the multiplication and involution on $M_C(A)$ are separately $C$-strictly continuous.

\item
$\KK(M_C(X))\subset M_C(\KK(X))$.

\item
$M_C(X)$ is the $C$-strict completion of $X$.

\item
$M_C(X)$ is an $M(C)$-sub-bimodule of $M(X)$.
\end{enumerate}

The main purpose of relative multipliers is the following extension theorem \cite[Proposition~A.11]{dkq}:
Suppose $(X,A)$ and $(Y,B)$ are (nondegenerate) $C^*$-correspondences and
$\kappa\colon C\to M(A)$ and $\sigma\colon D\to M(B)$
are nondegenerate homomorphisms.
For any correspondence homomorphism $(\psi,\pi)\colon (X,A)\to (M_D(Y),M_D(B))$,
if there is a nondegenerate homomorphism $\lambda\colon C\to M(\sigma(D))$
such that
\[
\pi(\kappa(c)a)=\lambda(c)\pi(a)\quad\text{for all $c\in C$ and $a\in A$,}
\]
then  there is a unique $C$-strict to $D$-strictly continuous correspondence homomorphism $(\bar\psi,\bar\pi)\colon (M_C(X),M_C(A))\to (M_D(Y),M_D(B))$ that extends $(\psi,\pi)$.

A closely related concept is the following, due to Baaj and Skandalis \cite{BaajSkandalis}.

\begin{defn}
For an ideal $I$ of a $C^*$-algebra $A$, let
\[
\ideal{A}{I}=\{m\in M(A)\mid mA\cup Am\subset I\}.
\]
\end{defn}

\begin{lem}\label{strict}
If $I$ is an ideal of a $C^*$-algebra $A$, then:
\begin{enumerate}
\item $\ideal{A}{I}$ is the strict closure of $I$ in $M(A)$;

\item if $\pi\colon A\to M(B)$ is a nondegenerate homomorphism such that $\pi(I)\subset B$, then
$\bar\pi(\ideal{A}{I})\subset M_A(B)$, and
$\bar\pi|\colon\ideal{A}{I}\to M_A(B)$
is strict to $A$-strictly continuous.
\end{enumerate}
\end{lem}

\begin{proof}
This is elementary, and the techniques are similar to those of \cite[Appendix~A]{dkq}.
For (i), we first show that $I$ is strictly dense in $\ideal{A}{I}$. Let $m\in \ideal{A}{I}$, and let $\{e_i\}$ be an approximate identity for $A$. Then $me_i\in I$ for all $i$, and 
$me_i\to m$ strictly in $M(A)$.

To see that $\ideal{A}{I}$ is strictly closed in $M(A)$, let $\{m_i\}$ be a net in $\ideal{A}{I}$ converging strictly to $m$ in $M(A)$. We must show that $m\in \ideal{A}{I}$. For $a\in A$ we have
\[
\|m_ia-ma\|\to 0\midtext{and}\|am_i-am\|\to 0,
\]
so $ma,am\in I$ because $m_ia,am_i\in I$ for all $i$.

For (ii), note that by (i) it suffices to show that $\pi|\colon I\to B$ is continuous for the relative strict topology of $I$ in $M(A)$ and the relative $A$-strict topology of $B\subset M_A(B)$.
Let $\{c_i\}$ be a net in $I$ converging strictly to $0$ in $M(A)$,
and let $a\in A$.
Then $\pi(c_i)\pi(a)=\pi(c_ia)$ and $\pi(a)\pi(c_i)=\pi(ac_i)$ converge to $0$ in norm, so $\pi(c_i)\to 0$ $A$-strictly in~$B$.
\end{proof}

\begin{rem}
Let $I$ be an ideal of a $C^*$-algebra $A$, and let $\rho\colon A\to M(I)$ be the canonical homomorphism. Then by \lemref{strict} the canonical extension $\bar\rho\colon M(A)\to M(I)$ maps $\ideal{A}{I}$ into $M_A(I)$.
In fact, the restriction $\bar\rho|_{\ideal{A}{I}}$ is injective, although we will not need this here.
\end{rem}

\section{Functoriality} \label{sec:functoriality}

\begin{defn}\label{CP correspondence homomorphism}
Let $(X,A)$ and $(Y,B)$ be nondegenerate $C^*$-correspondences. 
A correspondence homomorphism $(\psi,\pi)\colon (X,A)\to (M(Y),M(B))$ is \emph{Cuntz-Pimsner covariant} if
\begin{enumerate}
\item $\psi(X)\subset M_B(Y)$,

\item $\pi\colon A\to M(B)$ is nondegenerate,

\item $\pi(J_X)\subset \ideal{B}{J_Y}$, and

\item the diagram
\begin{equation*}
\xymatrix{
J_X \ar[r]^-{\pi|} \ar[d]_{\varphi_A|}
&\ideal{B}{J_Y} \ar[d]^{\bar{\varphi_B}\bigm|}
\\
\KK(X) \ar[r]_-{\cpct\psi}
&M_B(\KK(Y))
}
\end{equation*}
commutes, where $\cpct\psi$ is the homomorphism provided by Proposition~\ref{compacts}. 
\end{enumerate}
\end{defn}

The above definition simplifies when the correspondence homomorphism is nondegenerate:

\begin{lem}\label{lem:cpdefn}
A nondegenerate correspondence homomorphism 
$(\psi,\pi)\colon (X,A)\to (M(Y),M(B))$ is {Cuntz-Pimsner covariant} if
and only if items~\textup(i\textup) and~\textup(iii\textup) hold in Definition~\ref{CP correspondence homomorphism}.
\end{lem}

The lemma follows immediately from the following elementary result.
Lemma~\ref{lem:nondegcomm} is presumably well-known, but we could not find a reference
in the literature.

\begin{lem}\label{lem:nondegcomm}
For any nondegenerate correspondence homomorphism
$(\pi,\psi,\rho)\colon (A,X,B)\to (M(C),M(Y),M(D))$,
the diagram
\begin{equation*}
\xymatrix{
A \ar[rr]^-\pi \ar[d]_{\varphi_A}
&&M(C) \ar[d]^{\bar{\varphi_C}}
\\
\LL(X) \ar[rr]_-{\bar{\psi^{(1)}}}
&&\LL(Y)
}
\end{equation*}
commutes.
\end{lem}

\begin{proof}
Fix $a\in A$; we must show that $\bar{\psi^{(1)}}(\varphi_A(a))\eta
= \bar{\varphi_C}(\pi(a))\eta$ for all $\eta\in Y$.
By nondegeneracy it suffices to consider
elements of the form $\eta=\psi(\xi)d$ with $\xi\in X$ and $d\in D$,
in that case we have
\begin{align*}
\bar{\psi^{(1)}}(\varphi_A(a))\eta
&= \bar{\psi^{(1)}}(\varphi_A(a))\psi(\xi)d
= \psi(\varphi_A(a)\xi)d
= \psi(a\cdot\xi)d\\
&= \pi(a)\cdot\psi(\xi)d
= \bar{\varphi_C}(\pi(a))\psi(\xi)d
= \bar{\varphi_C}(\pi(a))\eta.\qedhere
\end{align*}
\end{proof}

The following lemma addresses the overlap between 
Cuntz-Pimsner covariant correspondence homomorphisms
and Cuntz-Pimsner covariant Toeplitz representations:

\begin{lem}\label{consistent}
Let $(X,A)$ be a nondegenerate $C^*$-correspondence,
and let $(\psi,\pi)\colon\(X,A)\to M(B)$ be a Toeplitz representation of $(X,A)$ 
in a $C^*$-algebra~$B$.
Then $(\psi,\pi)$ is Cuntz-Pimsner covariant 
as a correspondence homomorphism into $(B,B)$
\textup(as in \defnref{CP correspondence homomorphism}\textup)
if and only if $\pi\colon A\to M(B)$ is nondegenerate and $(\psi,\pi)$ is Cuntz-Pimsner covariant as a Toeplitz representation.
In particular, $(k_X,k_A)\colon(X,A)\to (\OO_X,\OO_X)$ is Cuntz-Pimsner covariant in the sense of \defnref{CP correspondence homomorphism}.
\end{lem}

\begin{proof}
This follows from the identifications
$J_B=\KK(B)=B$ and
$M_B(B)=M(B;B)=M_B(\KK(B))=M(B)$,
and the observation that $\varphi_B$ is the inclusion $B\hookrightarrow M(B)$.
\end{proof}

The Cuntz-Pimsner covariant homomorphisms between correspondences are the morphisms in a suitable category, which we now define.

\begin{thm}\label{is a category}
There is a category \cpcat\ that has:
\begin{itemize}
\item
nondegenerate $C^*$-correspondences as \emph{objects}, and

\item
Cuntz-Pimsner covariant homomorphisms $(\psi,\pi)\colon(X,A)\to (M(Y),M(B))$ 
\textup(as in \defnref{CP correspondence homomorphism}\textup) as \emph{morphisms} from $(X,A)$ to $(Y,B)$; 

\item
and in which the \emph{composition} of $(\psi,\pi)\colon(X,A)\to (Y,B)$ and $(\sigma,\tau)\colon(Y,B)\to (Z,C)$ 
is $\bigl(\bar\sigma\circ\psi,\bar\tau\circ\pi)$.
\end{itemize}
\end{thm}

\begin{proof}
First of all, to see that composition is well-defined, note that since $\tau\colon B\to M(C)$ is nondegenerate by definition, it follows from \cite[Proposition~A.11]{dkq} that $\sigma\colon Y\to M_C(Z)$ extends uniquely to a $B$-strict to $C$-strictly continuous homomorphism
\[
\bigl(\bar\sigma,\bar\tau\bigr)\colon (M_B(Y),M(B))\to (M_C(Z),M(C)).
\]
Thus we get a correspondence homomorphism
\[
(\bar\sigma\circ\psi,\bar\tau\circ\pi)\colon (X,A)\to (M_C(Z),M(C)),
\]
which we must check is Cuntz-Pimsner covariant in the sense of \defnref{CP correspondence homomorphism}. Certainly $\bar\tau\circ\pi\colon A\to M(C)$ is nondegenerate, so it remains to verify items (iii)--(iv) in \defnref{CP correspondence homomorphism}.

For (iii), since $\tau$ is nondegenerate we have
\begin{align*}
(\bar\tau\circ\pi)(J_X)C
&=\bar\tau(\pi(J_X))\tau(B)C\subset \bar\tau\bigl(\ideal{B}{J_Y}\bigr)\tau(B)C\\
&=\tau\bigl(\ideal{B}{J_Y}B\bigr)C\subset \tau(J_Y)C\subset J_Z,
\end{align*}
and similarly $C(\bar\tau\circ\pi)(J_X)\subset J_Z$.
Therefore $\bar\tau\circ\pi(J_X)\subset \ideal{C}{J_Z}$.

For (iv), let $a\in J_X$. We must show that
\[
\bar{\varphi_C}\circ (\bar\tau\circ\pi)(a)=(\bar\sigma\circ\psi)^{(1)}\circ\varphi_A(a)
\]
in $\LL(Z)$,
and by nondegeneracy it suffices to show equality after multiplying on the right by $\varphi_C(c)$ for an arbitrary $c\in C$. Again by nondegeneracy we can factor $c=\tau(b)c'$ for some $b\in B$ and $c'\in C$, and then
\begin{align*}
\bar{\varphi_C}\circ (\bar\tau\circ\pi)(a)\varphi_C(c)
&=\bar{\varphi_C}(\bar\tau(\pi(a))))\varphi_C\bigl(\tau(b)c'\bigr)
\\&=\bar{\varphi_C}\circ\tau\bigl(\pi(a)b\bigr)\varphi_C(c')
\\&=\sigma^{(1)}\circ\varphi_B\bigl(\pi(a)b\bigr)\varphi_C(c')
\\&=\sigma^{(1)}\bigl(\bar{\varphi_B}\circ\pi(a)\varphi_B(b)\bigr)\varphi_C(c')
\\&=\sigma^{(1)}\bigl(\psi^{(1)}\circ\varphi_A(a)\varphi_B(b)\bigr)\varphi_C(c')
\\&=\bar{\sigma^{(1)}}\circ\psi^{(1)}\circ\varphi_A(a)\sigma^{(1)}\circ\varphi_B(b)\varphi_C(c')
\\&=\bigl(\bar\sigma\circ\psi\bigr)^{(1)}\circ\varphi_A(a)\varphi_C\circ\tau(b)\varphi_C(c')
\\&=\bigl(\bar\sigma\circ\psi\bigr)^{(1)}\circ\varphi_A(a)\varphi_C\bigl(\tau(b)c'\bigr)
\\&=\bigl(\bar\sigma\circ\psi\bigr)^{(1)}\circ\varphi_A(a)\varphi_C(c).
\end{align*}
We have thus verified that composition is well-defined in \cpcat.

To see that composition is associative is a routine exercise in the definitions and the properties of ``barring'' (see, e.g., \cite[Appendix~A]{boil}): if also $(\zeta,\rho)\colon(Z,C)\to (W,D)$ in \cpcat\ then
\begin{align*}
(\zeta,\rho)\circ\bigl((\sigma,\tau)\circ(\psi,\pi)\bigr)
&=(\zeta,\rho)\circ(\bar\sigma\circ\psi,\bar\tau\circ\pi)
\\&=\bigl(\bar\zeta\circ(\bar\sigma\circ\psi),\bar\rho\circ(\bar\tau\circ\pi)\bigr)
\\&=\bigl((\bar\zeta\circ\bar\sigma)\circ\psi,(\bar\rho\circ\bar\tau)\circ\pi\bigr)
\\&=\Bigl(\bar{\bar\zeta\circ\sigma}\circ\psi,\bar{\bar\rho\circ\tau}\circ\pi\Bigr)
\\&=\bigl(\bar\zeta\circ\sigma,\bar\rho\circ\tau\bigr)\circ(\psi,\pi)
\\&=\bigl((\zeta,\rho)\circ(\sigma,\tau)\bigr)\circ(\psi,\pi).
\end{align*}

It is now clear that $(\id_X,\id_A)$ is an identity morphism on each object $(X,A)$, and therefore \cpcat\ is a category.
\end{proof}

\begin{cor} \label{functor}
Let $(X,A)$ and $(Y,B)$ be nondegenerate $C^*$-correspondences, and let $(\psi,\pi)\colon(X,A)\to (M(Y),M(B))$ be a Cuntz-Pimsner covariant correspondence homomorphism. Then there is a unique homomorphism $\OO_{\psi,\pi}$ making the diagram
\[
\xymatrix@C+30pt{
(X,A) \ar[r]^-{(\psi,\pi)} \ar[d]_{(k_X,k_A)}
&(M_B(Y),M(B)) \ar[d]^{(\bar{k_Y},\bar{k_B})}
\\
\OO_X \ar[r]_-{\OO_{\psi,\pi}}
&M_B(\OO_Y)
}
\]
commute.
Moreover, $\OO_{\psi,\pi}$ is nondegenerate, and is injective if $\pi$ is.
\end{cor}

\begin{proof}
Applying Theorem~\ref{is a category} with $(Z,C)$ being the $C^*$-algebra $\OO_Y$ viewed as a correspondence over itself in the canonical way, we see that
$(\bar{k_Y}\circ\psi,\bar{k_B}\circ\pi)$ is a Cuntz-Pimsner covariant Toeplitz representation $(\sigma,\nu)\colon X\to M_B(\OO_Y)$.
Then the universal property of $\OO_X$ gives the unique homomorphism
\[
\OO_{\psi,\pi}=\bigl(\bar{k_Y}\circ\psi\bigr)\times\bigl(\bar{k_B}\circ\pi\bigr).
\]

Nondegeneracy of $\OO_{\psi,\pi}$ follows from 
nondegeneracy of $\pi$ and $k_B$, since
\begin{align*}
\OO_Y 
&= \bar{k_B}\circ\pi(A)\OO_Y
= \OO_{\psi,\pi}\circ k_A(A)\OO_Y\\
&= \OO_{\psi,\pi}(k_A(A))\OO_Y
\subseteq \OO_{\psi,\pi}(\OO_X)\OO_Y \subseteq \OO_Y
\end{align*}
implies equality throughout.

If $\pi$ is injective, then so is $\bar{k_B}\circ\pi$, so to show $\OO_{\psi,\pi}$ is injective we can apply the Gauge-Invariant Uniqueness Theorem \cite[Theorem~6.4]{KatsuraCorrespondence}:
let $\gamma\colon\T\to\aut\OO_Y$ be the gauge action.
It suffices to observe that for all $z\in\T$, $\xi\in X$, and $a\in A$ we have
\begin{align*}
\bar{\gamma_z}\circ\bar{k_Y}\circ\psi(\xi)&=z\bar{k_Y}\circ\psi(\xi)\\
\bar{\gamma_z}\circ\bar{k_B}\circ\pi(a)&=\bar{k_B}\circ\pi(a).
\qedhere
\end{align*}
\end{proof}

Recall from, \emph{e.g.}, \cite{boil}, that there is a category \csnd\ that has:
\begin{itemize}
\item $C^*$-algebras  as \emph{objects}, and

\item nondegenerate homomorphisms $\pi:A\to M(B)$ as \emph{morphisms} from $A$ to $B$;

\item and in which the \emph{composition} of $\pi\colon A\to B$ and $\tau\colon B\to C$ is $\bar\tau\circ\pi$.
\end{itemize}

\begin{thm}\label{functor prove}
The assignments $X\mapsto \OO_X$ and $(\psi,\pi)\mapsto \OO_{\psi,\pi}$ define a functor from \cpcat\ to \csnd.
\end{thm}

\begin{proof}
First of all, it follows from Corollary~\ref{functor} that if $(\psi,\pi)\colon (X,A)\to (Y,B)$ in \cpcat\ then $\OO_{\psi,\pi}$ is a morphism from $\OO_X$ to $\OO_Y$ in \csnd.
Moreover, $\OO_{\id_X,\id_A}=\id_{\OO_X}$  by uniqueness.

To see that compositions are preserved, let $(\psi,\pi)\colon(X,A)\to (Y,B)$ and $(\sigma,\tau)\colon(Y,B)\to (Z,C)$ in 
\cpcat.
We have
\begin{align*}
\OO_{(\sigma,\tau)\circ(\psi,\pi)}\circ &k_X
=\bar{k_Z}\circ(\bar\sigma\circ\psi)=\bar{k_Z}\circ\bar\sigma\circ\psi\\
&=\bar{\OO_{\sigma,\tau}}\circ\bar{k_Y}\circ\psi=\bar{\OO_{\sigma,\tau}}\circ\OO_{\psi,\pi}\circ k_X,
\end{align*}
and similarly
\[
\OO_{(\sigma,\tau)\circ(\psi,\pi)}\circ k_A=\bar{\OO_{\sigma,\tau}}\circ\OO_{\psi,\pi}\circ k_A,
\]
so that $\OO_{(\sigma,\tau)\circ(\psi,\pi)}=\OO_{\sigma,\tau}\circ\OO_{\psi,\pi}$ in \csnd.
\end{proof}

\section{Applications}\label{applications}

We give three applications of \corref{functor}.

\subsection*{Topological graph actions}

Our first application is historical; we show that in \cite{dkq} the germ of the idea of \corref{functor} was introduced in an ad-hoc way, in a very special case.

Theorem~5.6 of \cite{dkq} shows that if a locally compact group $G$ acts freely and properly on a topological graph~$E$, then the quotient $E/G$ is also a topological graph, and $C^*(E)\rtimes_r G$ and $C^*(E/G)$ are Morita equivalent.
The strategy for proving Morita equivalence in \cite{dkq} is to construct an isomorphism of $C^*(E/G)$ with Rieffel's generalized fixed-point algebra $C^*(E)^G$, after showing that the action of $G$ on $C^*(E)$ is saturated and proper in the sense of~\cite{proper}, and then appealing to the imprimitivity theorem $C^*(E)\rtimes_r G\sim_M C^*(E)^G$ of \cite{proper}.

By definition, $C^*(E)^G$ is a $C^*$-subalgebra of $M(C^*(E))$, and
the isomorphism of $C^*(E/G)$ onto $C^*(E)^G$ is constructed from a Cuntz-Pimsner covariant Toeplitz representation $(\tau,\pi)$ of the topological-graph correspondence $(X(E/G),C_0((E/G)^0))$ in $M(C^*(E))$,
which in turn is constructed via a correspondence homomorphism $(\mu,\nu)$ from $(X(E/G),C_0((E/G)^0))$ to $(M(X(E)),M(C^0(E^0)))$.
The proof in \cite{dkq} that $(\tau,\pi)$ is Cuntz-Pimsner covariant 
essentially uses a special case of the concept of Cuntz-Pimsner covariant correspondence homomorphisms defined in 
 \defnref{CP correspondence homomorphism}.
 
We will now explain this in more detail.
First, recall that the $C_0(E^0)$-correspondence $X(E)$ is a completion of $C_c(E^1)$,
the Katsura ideal $J_{X(E)}$ can be identified with $C_0(E^0_{\rg})$, where $E^0_{\rg}$ is a certain open subset of $E^0$,
and the topological-graph algebra $C^*(E)$ is the Cuntz-Pimsner algebra $\OO_{X(E)}$.
It will help the exposition to introduce the following temporary notation:
\begin{itemize}
\item $A=C_0((E/G)^0)$
\item $X=X((E/G))$
\item $A_{\rg}=J_X$
\item $B=C_0(E^0)$
\item $Y=X(E)$
\item $B_{\rg}=J_Y$
\item $q:E\to E/G$ is the quotient map (both for edges $E^1$ and vertices $E^0$).
\end{itemize}
(Warning: the roles of $X,A$ and $Y,B$ between \cite{dkq} and here are switched, to allow more convenient reference to the methods of the current paper.)
\cite{dkq} constructed the correspondence homomorphism
\[
(\mu,\nu):(X,A)\to (M_B(Y),M(B))
\]
starting with
\begin{itemize}
\item $\nu(f)=f\circ q$
\item $(\mu(\xi)\cdot g)(e)=\xi(q(e))g(e)$ for $\xi\in C_c((E/G)^1)$, $g\in C_c(E^0)$, and $e\in E^1$.
\end{itemize}
Then the pair $(X,A)$ was mapped into $M(\OO_Y)$ in \cite{dkq} by the commutative diagram
\[
\xymatrix@C+30pt{
(X,A) \ar[r]^-{(\mu,\nu)} \ar[dr]_{(\tau,\pi)}
&(M_B(Y),M(B)) \ar[d]^{(\bar{k_Y},\bar{k_B})}
\\
&M_B(\OO_Y)
}
\]
In \cite{dkq} it was then recognized that Cuntz-Pimsner covariance of $(\tau,\pi)$ is expressed by commutativity of the left-hand triangle of the following diagram, which is a version of \cite[page~1547, diagram~(5)]{dkq} in which some of the notation has been modified to be consistent with the current paper:
\begin{equation}\label{commute}
\xymatrix{
J_X \ar[rr]^{\nu'} \ar[dr]^{\pi|} \ar[dd]_{\varphi_A|}
&&M_B(J_Y) \ar[dl]_{\bar{k_B|}} \ar[dd]^{\bar{\varphi_B|}}
\\
&M_B(\OO_Y)
\\
\KK(X) \ar[ur]^{\tau^{(1)}} \ar[rr]_{\mu^{(1)}}
&&M_B(\KK(Y)), \ar[ul]_{\bar{k_Y^{(1)}}}
}
\end{equation}
and the strategy was to verify that the other parts of the diagram commute.
Here the homomorphism $\nu'$ is constructed from the commutative diagram
\[
\xymatrix{
J_X \ar[rr]^{\nu|} \ar[dr]_{\nu'}
&&M(B;J_Y) \ar@{_(->}[dl]
\\
&M_B(J_Y),
}
\]
where the inclusion $M(B:J_Y)\hookrightarrow M_B(J_Y)$ is given by restriction
\[
g\mapsto g|_{Y^0_{\rg}}.
\]
In \cite{dkq} it was not recognized that in fact it would be better to do away with the map $\nu'$ altogether, so that the outer square of \eqref{commute} is replaced by
\[
\xymatrix{
J_X \ar[r]^-{\nu|} \ar[d]_{\varphi_A|}
&M(B;J_Y) \ar[d]^{\bar{\varphi_B}|}
\\
\KK(X) \ar[r]_-{\mu^{(1)}}
&M_B(\KK(Y)),
}
\]
whose commutativity is precisely our definition \defnref{CP correspondence homomorphism} of Cuntz-Pimsner covariance of the correspondence homomorphism $(\mu,\nu)$.
The computations in \cite{dkq} were much more painstakingly ``bare-hands'' than in the current paper, because, again, the techniques were entirely ad-hoc, whereas here we take a more conceptual and systematic approach, developing appropriate machinery along the way.

\subsection*{Topological graph coactions}

In a forthcoming paper \cite{tgcoact}, a continuous map (``cocycle'') $\kappa$ on a topological graph $E$ with values in a locally compact group $G$ is used to construct a coaction~$\delta$ of~$G$ on~$C^*(E)$,
with an eye toward proving that the crossed product $C^*(E)\rtimes_\delta G$ is isomorphic to the $C^*$-algebra of the skew-product topological graph $E\rtimes_\kappa G$,
thereby generalizing \cite[Theorem~2.4]{kqrskew} from the discrete case.

To describe this application, let
\begin{itemize}
\item $A=C_0(E^0)$
\item $X=X(E)$,
\end{itemize}
and recall, e.g.,  from the discussion at the beginning of this section, that $C^*(E)$ is the Cuntz-Pimsner algebra of the $C^*$-correspondence $(X,A)$.

To be a coaction, $\delta$ must in particular
be a homomorphism from $C^*(E)$ to $M_{1\otimes C^*(G)}(C^*(E)\otimes C^*(G))$.
As usual, $\delta$ is constructed from a Cuntz-Pimsner covariant Toeplitz representation 
$(\delta_X,\delta_{A})$ 
of $(X,A)$ in the $C^*$-algebra $M(C^*(E)\otimes C^*(G))$,
which in turn is constructed via a correspondence homomorphism $(\sigma,\id_A\otimes 1)$
from $(X,A)$ to $(M(X\otimes C^*(G)),M(A\otimes C^*(G)))$.
Techniques based upon 
\corref{functor} will be used to 
show that the correspondence homomorphsim $(\sigma,\id_A\otimes 1)$ gives rise to a coaction $\delta$ on $C^*(E)$.
Interestingly, however, we will need a slight strengthening of the Cuntz-Pimsner covariance condition of \defnref{CP correspondence homomorphism}.
The problem is that, due to nonexactness of minimal $C^*$-tensor products,
we have no reason to believe that the minimal tensor product $\OO_X\otimes C^*(G)$ coincides with the Cuntz-Pimsner algebra $\OO_{X\otimes C^*(G)}$ of the external-tensor-product correspondence (where $C^*(G)$ is regarded as a correspondence over itself in the standard way).
In fact, the basic theory of coactions on Cuntz-Pimsner algebras will require a significant amount of work, which we will do in \cite{KQRCorrespondenceCoaction}.

Anyway, once we have the machinery necessary to construct coactions on Cuntz-Pimsner algebras, our application in \cite{tgcoact} will go roughly as follows:
first of all, a $G$-valued cocycle on a topological graph $E$ is just a continuous map $\kappa:E^1\to G$.
Since the $A$-correspondence $X$ is a completion of $C_c(E^1)$, and the group $G$ embeds as unitary multipliers of $C^*(G)$, we are led to regard the cocycle $\kappa$ as an adjointable operator $v$ on $X\otimes C^*(G)$, and then we are able to define a coaction on $X$ by $\sigma(\xi)=v(\xi\otimes 1)$, where $\xi\otimes 1$ is regarded as a multiplier of the correspondence $X\otimes C^*(G)$.
This is a continuous version of the coaction $\Chi_{\{e\}}\mapsto \Chi_{\{e\}}\otimes \kappa(e)$ of \cite{kqrskew}.
We emphasize that the justification that this actually gives a coaction will depend upon the preparation to come in \cite{KQRCorrespondenceCoaction}.

\subsection*{$C^*$-correspondence action crossed products.}

Let $(\gamma,\alpha)$ be an action of a locally compact group $G$ on a nondegenerate correspondence $(X,A)$. 
The \emph{crossed product} is the completion $(X\rtimes_\gamma G,A\rtimes_\alpha G)$ of the pre-correspondence $(C_c(G,X),C_c(G,A))$ with operations
\begin{align*}
(f\cdot\xi)(s)&=\int_G f(t)\cdot\gamma_t\bigl(\xi(t\inv s)\bigr)\,dt\\
(\xi \cdot f)(s)&=\int_G \xi(t)\cdot\alpha_t\bigl(f(t^{-1}s)\bigr)\,dt\\
\inn \xi \eta (s)&=\int_G \alpha_{t^{-1}}\bigl(\bigl\<\xi(t),\eta(ts)\bigr\>\bigr)\,dt
\end{align*}
for $f,g\in C_c(G,A)$ and $\xi,\eta\in C_c(G,X)$.
(We refer to, e.g., 
\cite{taco}, 
\cite{HN},
\cite[Chapters~2 and 3]{enchilada}, and
\cite{K}
for the elementary theory of actions and crossed products for correspondences.)

Since $(\gamma_s,\alpha_s)\colon (X,A)\to (X,A)$ is a correspondence homomorphism
for each $s\in G$, Proposition~\ref{compacts} provides 
homomorphisms $\gamma_s^{(1)}\colon\KK(X)\to\KK(X)$,
which by uniqueness give rise to an action $(\gamma^{(1)},\gamma,\alpha)$ of $G$ on the correspondence $(\KK(X),X,A)$, and such that $\varphi_A\colon A\to M(\KK(X))$ is $\alpha-\gamma^{(1)}$ equivariant.

There is an isomorphism (see, \emph{e.g.}, \cite[3.11]{K})
\[
\tau\colon\KK(X\rtimes_\gamma G)\iso \KK(X)\rtimes_{\gamma^{(1)}} G
\]
satisfying
\[
 \tau(\theta_{\xi,\eta})(s) = \int_G \theta_{\xi(t),\gamma_s(\eta(s^{-1}t))} \Delta(s^{-1}t) dt,
\]
where $\xi, \eta \in C_c(G,X), s \in G$ and $\Delta$ is the modular function of $G$,
and moreover the diagram
\[
\xymatrix@C+30pt{
A\rtimes_\alpha G \ar[r]^-{\varphi_{A\rtimes_\alpha G}} \ar[dr]_{\varphi_A\rtimes G}
&M\bigl(\KK(X\rtimes_\gamma G)\bigr) \ar[d]^{\bar\tau}_\cong
\\
&M\bigl(\KK(X)\rtimes_{\gamma^{(1)}} G\bigr)
}
\]
commutes.

Let $(i_A,i_G)\colon(A,G)\to M(A\rtimes_\alpha G)$ be the canonical covariant homomorphism.

\begin{prop} \label{prop:morphism}
There is a Cuntz-Pimsner covariant correspondence homomorphism
\[
(i_X,i_A)\colon(X,A)\to \bigl(M(X \rtimes_\gamma G), M(A \rtimes_\alpha G)\bigr)
\]
such that for $x \in X, f \in C_c(G,A), s \in G$ we have
\begin{equation}\label{i_X}
\bigl(i_X(x)\cdot f\bigr)(s) = x\cdot f(s).
\end{equation}
\end{prop}

\begin{proof}
We first claim that
for fixed $x \in X$, \eqref{i_X} uniquely determines
an operator $i_X(x)\colon A \rtimes_\alpha G \to X \rtimes_\gamma G$ with adjoint given by
\[
 \bigl(i_X(x)^*\xi\bigr)(s) = \inn x {\xi(s)}
\]
for $\xi \in C_c(G,X)$ and $s \in G$. 
Indeed,
\eqref{i_X} certainly defines a right $C_c(G,A)$-module map $C_c(G,A)\to C_c(G,X)$, and we can check the adjoint property on generators:
for any $\xi \in C_c(G,X)$, $f \in C_c(G,A)$, and $s \in G$ we simply calculate
\begin{align*}
\bigl\<i_X(x)f,\xi\bigr\>(s)
&=\int_G \alpha_{t^{-1}}\bigl(\bigl\<(i_X(x)f)(t),\xi(ts)\bigr\>\bigr)\,dt
\\&=\int_G \alpha_{t\inv}\bigl(f(t)^*\<x,\xi(ts)\>\bigr)\,dt
\\&=\int_G \alpha_t\bigl(f(t\inv)\Delta(t\inv)\<x,\xi(t\inv s)\>\bigr)\,dt
\\&=\int_G f^*(t)\alpha_t\bigl(\<x,\xi(t\inv s)\>\bigr)\,dt,
\end{align*}
proving the claim.

It is now straightforward to verify that the pair $(i_X,i_A)$ is a correspondence
homomorphism. For example, for $x,y \in X, f \in C_c(G,A)$ and $s \in G$ we calculate
\begin{align*}
 \bigl(\inn {i_X(x)} {i_X(y)} f\bigr)(s) 
 &= \bigl(i_X(x)^*i_X(y) \cdot f\bigr)(s) 
= \bigl\< x, (i_X(y)\cdot f(s)\bigr\>\\
 &= \inn x {y\cdot f(s)} 
 = \inn x y f(s) 
 = \bigl(i_A(\inn x y)f\bigr)(s)
\end{align*}
as required.

To show that $(i_X,i_A)$ is Cuntz-Pimsner covariant, 
by Lemma~\ref{lem:cpdefn} it suffices to show 
that $(i_X,i_A)$ is nondegenerate
and satisfies items~(i) and~(iii) in Definition~\ref{CP correspondence homomorphism}.

We already know that the coefficient map $i_A$ is nondegenerate. 
Then for $x \in X$, $a\in A$, $g\in C_c(G)$ we have
\[
i_X(x)\cdot \bigr(i_A(a)i_G(g)\bigr)
=i_X(x\cdot a)i_G(g),
\]
so nondegeneracy of $(i_X,i_A)$ follows since $X\cdot A=X$ and
\[
X\rtimes_\gamma G=\bar{i_X(X)\cdot i_G(C_c(G))}.
\]

Item~(i),
that $i_X(X) \subset M_{A \rtimes_\alpha G}(X \rtimes_\gamma G)$, is clear from the definition.
To verify item~(iii), that $i_A(J_X) \subset \ideal{A \rtimes_\alpha G}{J_{X \rtimes_\gamma G} }$, fix $a \in J_X$. Firstly, \cite[Lemma 2.6(a)]{HN} says that $J_X$ is $\alpha$-invariant, and we know from \cite[Proposition 2.7]{HN} that $J_X \rtimes_\alpha G \subset J_{X \rtimes_\gamma G}$. For any $f \in C_c(G,A)$ and $s \in G$ we have
\[
 \bigl(i_A(a)f\bigr)(s) =a f(s) \in J_X
\]
so $i_A(a)f \in C_c(G,J_X)$. Hence we get
\[
 i_A(J_X)(A \rtimes_\alpha G) \subset J_X \rtimes_\alpha G \subset J_{X \rtimes_\gamma G}
\]
as required.
\end{proof}

\begin{prop}
With $i_X$ as in \propref{prop:morphism}, let
\[
\OO_{i_X,i_A}\colon\OO_X\to M(\OO_{X\rtimes_\gamma G})
\]
be the nondegenerate homomorphism vouchsafed by \corref{functor}.
Also define a strictly continuous unitary homomorphism
$u\colon G\to M(\OO_{X\rtimes_\gamma G})$ by
\[
u=\bar{k_{A\rtimes_\alpha G}}\circ i_G.
\]
Then
the pair $(\OO_{i_X,i_A},u)$ defines a covariant homomorphism of the $C^*$-dynamical system $(\OO_X,G,\beta)$ in the Cuntz-Pimsner algebra $\OO_{X \rtimes_\gamma G}$. 
\end{prop}

\begin{proof}
We only need to verify the covariance condition, namely that for each $t\in G$ we have
\[
 \ad u(t) \circ \OO_{i_X,i_A} = \OO_{i_X,i_A}\circ \beta_t,
\]
and it is enough to check this on the generators from $X$ and $A$.
For $x\in X$ we have
\begin{align*}
u(t)\OO_{i_X,i_A}\circ k_X(x)
&=\bar{k_{A\rtimes_\alpha G}}\bigl(i_G(t)\bigr)\bar{k_{X\rtimes_\gamma G}}\bigl( i_X(x)\bigr)
\\&=\bar{k_{X\rtimes_\gamma G}}\bigl(i_G(t)\cdot i_X(x)\bigr)
\\&=\bar{k_{X\rtimes _\gamma G}}\Bigl(i_X\bigl(\gamma_t(x)\bigr)\cdot i_G(t)\Bigr)
\\&=\bar{k_{X\rtimes_\gamma G}}\Bigl(i_X\bigl(\gamma_t(x)\bigr)\Bigr)\bar{k_{A\rtimes_\alpha G}}\bigl(i_G(t)\bigr)
\\&=\OO_{i_X,i_A}\circ k_X\circ\gamma_t(x)u(t)
\\&=\OO_{i_X,i_A}\circ\beta_t\circ k_X(x)u(t),
\end{align*}
and the calculation for generators from $A$ is similar.
\end{proof}

\begin{prop}\label{surjection}
The integrated form 
\[
 \OO_{i_X,i_A} \times u\colon\OO_X \rtimes_\beta G \to \OO_{X \rtimes_\gamma G}
\]
of the covariant pair $(\OO_{i_X,i_A}, u)$
is surjective.
\end{prop}

\begin{proof}
Let $f \in C_c(G,A)$. Then $k_A\circ f\in C_c(G,\OO_X)$, and
since $(i_A,i_G)$ is a covariant pair we have
\begin{align*}
 (\OO_{i_X,i_A} \times u)(k_A \circ f) &= \int_G \OO_{i_X,i_A}(k_A(f(t))) u(t) \,dt \\
 &= \overline{k_{A \rtimes_\alpha G}} \left(\int_G  i_A (f(t)) i_G(t) \,dt \right) \\
 &= k_{A \rtimes_\alpha G} \bigl((i_A \times i_G) (f)\bigr) \\
 &= k_{A \rtimes_\alpha G} (f).
\end{align*}
Therefore the image of $\OO_{i_X,i_A} \times u$ contains $k_{A \rtimes_\alpha G} (A \rtimes_\alpha G)$. 

Now let $\xi \in C_c(G,X)$. A similar calculation shows that
\[
 (\OO_{i_X,i_A} \times u)(k_X \circ \xi) 
 = \overline{k_{X \rtimes_\gamma G}} \left(\int_G  i_X (\xi(t))\cdot i_G(t) \,dt \right).
\]
Now, for any $f \in C_c(G,A), s \in G$ we can calculate
\begin{align*}
 \left(\int_G  i_X (\xi(t)) \cdot i_G(t) \,dt \,f\right) (s) 
 &= \left(\int_G  i_X (\xi(t))\cdot i_G(t)f \,dt \right)(s) \\
 &= \int_G \bigl( i_X (\xi(t))\cdot i_G(t)f\bigr)(s) \,dt \\
 &= \int_G \xi(t) \cdot \bigl(i_G(t) f\bigr)(s) \,dt \\
 &= \int_G \xi(t) \cdot\alpha_t(f(t^{-1}s)) \,dt \\
 &= (\xi \cdot f)(s),
\end{align*}
and so we have
\[
 \int_G  i_A (\xi(t))\cdot i_G(t) \,dt  = \xi .
\]
Thus the image of $\OO_{i_X,i_A} \times u$ also contains $k_{X \rtimes_\gamma G}(X \rtimes_\gamma G)$, and it now follows that $\OO_{i_X,i_A}$ is surjective.
\end{proof}

\begin{rem}
Ideally we would like the map in \propref{surjection} to be injective, which would give
 $\OO_X \rtimes_\beta G \cong \OO_{X \rtimes_\gamma G}$
for an arbitrary locally compact group $G$. In a forthcoming paper
\cite{KQRCorrespondenceCoaction}
we 
will use our techniques to prove this when the group $G$ is amenable, thereby recovering the result previously obtained by Hao and Ng \cite[Theorem~2.10]{HN} using the theory of correspondence coactions.
\end{rem}


\providecommand{\bysame}{\leavevmode\hbox to3em{\hrulefill}\thinspace}
\providecommand{\MR}{\relax\ifhmode\unskip\space\fi MR }
\providecommand{\MRhref}[2]{%
  \href{http://www.ams.org/mathscinet-getitem?mr=#1}{#2}
}
\providecommand{\href}[2]{#2}

\end{document}